\documentclass[hidelinks]{article}
\usepackage[utf8]{inputenc}
\usepackage{titlesec}
\usepackage[breaklinks=true]{hyperref}
\usepackage[includehead,a4paper,headheight=13.6pt]{geometry} 
\usepackage[indent=15pt]{parskip} 
\usepackage{amssymb,amsmath,amsthm}
\usepackage[nameinlink]{cleveref}
\usepackage{indentfirst}
\usepackage{courier}
\usepackage{pdfpages}

\newtheorem{thm}{Theorem}
\newtheorem{lem}[thm]{Lemma}

\titleformat*{\section}{\large\bfseries}
\titleformat*{\subsection}{\normalsize\bfseries}
\title{\vspace{-2em}Regularity of Singular Solutions to \(p\)-Poisson Equations}
\author{Sullivan Francis MacDonald\footnote{University of Toronto, ON, Canada. sullivan@math.utoronto.ca. The author is a PhD student supported by the University of Toronto Department of Mathematics.}}
\date{September 2023}

\begin{document}
\maketitle

\begin{abstract}
This work showcases level set estimates for weak solutions to the \(p\)-Poisson equation on a bounded domain, which we use to establish Lebesgue space inclusions for weak solutions. In particular we show that if \(\Omega\subset\mathbb{R}^n\) is a bounded domain and \(u\) is a weak solution to the Dirichlet problem for Poisson's equation
\[
    \begin{array}{rllc}
    -\Delta u=&\!\!\!\!f&\textrm{in}&\Omega\\
    u=&\!\!\!\!0&\textrm{on}&\partial\Omega
    \end{array}
\]
for \(f\in L^q(\Omega)\) with \(q<\frac{n}{2}\), then \(u\in L^r(\Omega)\) for every \(r<\frac{qn}{n-2q}\) and indeed \(\|u\|_r\leq C\|f\|_q\). This result is shown to be sharp, and similar regularity is established for solutions to the \(p\)-Poisson equation including in the edge case \(q=\frac{n}{p}\).
\end{abstract}

\section{Introduction}

In this work we are concerned with global regularity of weak solutions to Dirichlet problems for the \(p\)-Poisson equation on a bounded domain \(\Omega\subset\mathbb{R}^n\) for \(n\geq 3\),
\begin{equation}\label{p-Poisson}
    \begin{array}{rllc}
    -\mathrm{{div}} (| \nabla u|^{p-2}\nabla u)=&\!\!\!\!f&\textrm{in}&\Omega\\
    u=&\!\!\!\!0&\textrm{on}&\partial\Omega.
    \end{array}
\end{equation}
In particular we are interested in the behaviour of solutions when \(f\in L^q(\Omega)\) for \(q\leq \frac{n}{p}\).

In the case \(p=2\) it is a well known consequence of the De Giorgi-Nash-Moser theory that if \(f\in L^q(\Omega)\) for \(q>\frac{n}{2}\) then \(u\) is bounded and H\"older continuous up to \(\partial\Omega\), see \cite{Mouhot2018DEGA} and the references therein. Similar regularity is established for the \(p\)-Poisson equation with sufficiently regular data in \cite{DIBENEDETTO1983827}. In the case of \(f\in L^q(\Omega)\) for small \(q\) solutions to \eqref{p-Poisson} can be unbounded, though some regularity results have been established. For instance, a comparison principle for unbounded solutions is proved by Lenori \& Porretta \cite{LEONORI20181492}, and Lindqvist \cite[Thm 5.11]{Lindqvist} shows that \(p\)-superharmonic functions retain some local integrability even if they are not bounded.

There are applications to modelling and numerical schemes in which unbounded solutions to problem \eqref{p-Poisson} arise naturally, see e.g. \cite{ASHYRALIYEV200820,IJNAM-14-500}, and studying solutions in this setting also demands refinement of techniques which may be useful in work on related problems. It is therefore worthwhile to pursue improved regularity results for singular solutions to \eqref{p-Poisson}. It has bee shown that if \(p=2\) and \(f\in L^q(\Omega)\) for \(1<q\leq\infty\) then one has \(u\in W^{2,q}(\Omega')\) for some \(\Omega'\subset\Omega\) \cite[(1.4)]{TEIXEIRA2013150}, though it turns out that \(u\) inherits additional regularity from \(f\). Our main contribution in this note is to establish the following improved Lebesgue space inclusions for \(u\) when \(q\) is small.

\begin{thm}\label{main}
Let \(u\) solve \eqref{p-Poisson} for \(f\in L^q(\Omega)\) with \(1< q<\frac{n}{p}\). If \(r<\frac{(p-1)qn}{n-pq}\) then \(u\in L^r(\Omega)\) and
\begin{equation}\label{Lp Bound}
    \|u\|_r\leq C\|f\|_q^\frac{1}{p-1}.
\end{equation}
If \(r>\frac{(p-1)qn}{n-pq}\) then there exists \(f\in L^q(\Omega)\) and a solution \(u\) to \eqref{p-Poisson} such that \(u\not\in L^r(\Omega)\). If \(q=\frac{n}{p}\) then estimate \eqref{Lp Bound} holds for every \(r<\infty\).
\end{thm}

The techniques used to prove Theorem \ref{main} employ ideas from both Moser and De Giorgi iterative schemes, and they can be used to reproduce many classical boundedness results in the setting of Lebesgue and Orlicz spaces. Moreover, the sharpness of our exponent implies that this result is optimal on the scale of Lebesgue spaces, though it is unclear whether solutions necessarily belong to \(L^r(\Omega)\) when \(r=\frac{(p-1)qn}{n-qp}\). We also remark that the level set estimates used to prove our main result have been established for many elliptic equations, see e.g. \cite{dgref}, and we foresee no major difficulty in extending Theorem \ref{main} to such problems.

The remainder of this paper is organized as follows. In Section 2 we state some preliminary definitions and results, and in Section 3 we present a streamlined form of De Giorgi iteration to achieve an estimate for the distribution function of weak solutions. In Section 4 we prove \eqref{Lp Bound} using this distribution estimate together with an iterative argument inspired by Moser iteration. Finally, in Section 5, we present examples which show that our main result cannot be improved.

\section{Preliminaries}

For \(1\leq q<\infty\) we define \(L^q(\Omega)\) in the usual way, and if \(f\) is measurable we define the distribution function of \(f\) by \(\lambda_f(\alpha)=m(\{x\in\Omega:|f(x)|>\alpha\})\), where \(m\) is Lebesgue measure. Further, we remind the reader of the identity
\begin{equation}\label{dist}
    \int_\Omega |f|^qdx=\int_0^\infty q\alpha^{q-1}\lambda_f(\alpha)d\alpha,
\end{equation}
which follows from Tonelli's Theorem provided that \(f\) is measurable.

The Sobolev space \(W^{1,p}_0(\Omega)\) is defined as the collection of weakly differentiable functions on \(\Omega\) which vanish on \(\partial\Omega\), and whose weak derivatives belong to \(L^p(\Omega)\). Since \(\Omega\) is a bounded open set by assumption, if \(n>p\) then for \(f\in W^{1,p}_0(\Omega)\) we have the Sobolev inequality \(\|f\|_{\frac{np}{n-p}}\leq C\|\nabla f\|_p\).

A function \(u\in W^{1,p}_0(\Omega)\) is said to be a weak solution to the \(p\)-Poisson equation \eqref{p-Poisson} if 
\begin{equation}\label{weak}
    \int_\Omega|\nabla u|^{p-2}\nabla u\cdot\nabla\varphi dx=\int_\Omega f\varphi dx    
\end{equation}
holds for every \(\varphi\in C_0^\infty(\Omega)\). Indeed, thanks to a standard density argument, if \(u\in W^{1,p}_0(\Omega)\) is a weak solution to \eqref{p-Poisson} then equation \eqref{weak} holds for every \(\varphi\in W_0^{1,p}(\Omega)\). Henceforth then, we use \(W_0^{1,p}(\Omega)\) as our space of test functions with the understanding that all derivatives in \eqref{weak} are taken in the weak sense.

It is necessary to assume some additional regularity of \(f\) for the integral on the right-hand side of \eqref{weak} to be finite for an arbitrary test function. Since \(\varphi\in W_0^{1,p}(\Omega)\) belongs to \(L^\frac{np}{n-p}(\Omega)\) by Sobolev's inequality, if we assume going forward that \(f\in L^q(\Omega)\) for \(q\geq \frac{np}{np-n+p}\) then convergence of the integrals in \eqref{weak} is assured by H\"older's inequality.

\section{Distribution Estimates}

Given that the norm of a function \(f\) can be computed if one knows its distribution function \(\lambda_f\) using \eqref{dist}, we aim to estimate the measure of level sets for solutions to \eqref{p-Poisson} at an arbitrary height. We do this recursively by employing the following result.

\begin{lem}\label{distbound}
Let \(u\) be a weak solution to \eqref{p-Poisson} with \(1<p<n\) and let \(f\in L^q(\Omega)\) for \(1<q\leq\infty\). There exists a constant \(C\) such that for any \(\beta>\alpha\),
\[
    \lambda_u(\beta)\leq \bigg(\frac{C\|f\|_q\lambda_u(\alpha)^{\frac{q-1}{q}}}{(\beta-\alpha)^{p-1}}\bigg)^\frac{n}{n-p}.    
\]
\end{lem}

Our proof of this result employs a similar technique to De Giorgi iteration, though by using a modified test function we condense the standard argument. If one assumes that \(f\) belongs to the Orlicz space \(L^\Psi(\Omega)\) for any Young function \(\Psi\), the argument which we give below can be used to prove that
\[
    \lambda_u(\beta)\leq \bigg(\frac{C\|f\|_\Psi\overline{\Psi}^{-1}(\lambda_u(\alpha)^{-1})^{-1}}{(\beta-\alpha)^{p-1}}\bigg)^\frac{n}{n-p}.    
\]
This estimate allows one to prove boundedness of weak solutions under appropriate hypotheses.

\begin{proof}
Let \(u\in W_0^{1,p}(\Omega)\) be a weak solution to \eqref{p-Poisson} and fix \(\beta>\alpha\). Then define the test function
\[
    \varphi=(u-\alpha)_+-(u-\beta)_+
\]
so that \(\varphi\in W^{1,p}_0(\Omega)\) by the chain rule and a density argument. Moreover we have \(0\leq \varphi\leq \beta-\alpha\) and \(\varphi\) is nonzero only on the level set \(\{x\in\Omega:u(x)>\alpha\}\). For brevity we let \(\chi_\alpha\) denote the indicator of this set and we note that \(\nabla\varphi=\chi_\alpha\nabla u\) in the weak sense.

To estimate $\lambda_u(\beta)$ we begin by observing that
\begin{equation}\label{a}
    (\beta-\alpha)\lambda_u(\beta)^\frac{n-p}{pn}=\bigg(\int_{\{u>\beta\}}(\beta-\alpha)^\frac{pn}{n-p}dx\bigg)^\frac{n-p}{pn}\leq \bigg(\int_{\Omega}|\varphi|^\frac{pn}{n-p}dx\bigg)^\frac{n-p}{pn}=\|\varphi\|_\frac{pn}{n-p}.
\end{equation}
Our aim is to estimate the norm on the right-hand side in terms of \(\lambda_u(\alpha)\), to obtain a recursive estimate for \(\lambda_u\). Since $\varphi\in W^{1,p}_0(\Omega)$ we can use it as a test function in \eqref{weak} to get
\[
    \|\nabla\varphi\|_p^p=\int_\Omega|\nabla\varphi|^{p}dx=\int_\Omega|\nabla u|^{p}\chi_\alpha dx=\int_\Omega f\varphi dx.
\]
It follows from the Sobolev inequality, the estimate above, and H\"older's inequality that for \(q>1\),
\begin{equation}\label{b}
    \|\varphi\|_\frac{pn}{n-p}^p\leq C\int_\Omega f\varphi\chi_\alpha dx \leq C(\beta-\alpha)\int_\Omega|f|\chi_\alpha dx\leq C(\beta-\alpha)\|f\|_q\lambda_u(\alpha)^{1-\frac{1}{q}}.    
\end{equation}
Finally, combining equations \eqref{a} and \eqref{b} we get
\[
    (\beta-\alpha)\lambda_u(\beta)^\frac{n-p}{pn}\leq \|\varphi\|_\frac{pn}{n-p}\leq C(\beta-\alpha)^\frac{1}{p}\|f\|_q^\frac{1}{p}\lambda_u(\alpha)^{\frac{q-1}{pq}}.    
\]
Rearranging this gives the claimed distribution estimate.
\end{proof}

\section{Proof of Theorem \ref{main}}

For simplicity, we first take \(C\) as in Lemma \ref{distbound} and replace \(u\) with \(v=(C\|f\|_q)^{-\frac{1}{p-1}}u\) so that
\[
    \lambda_v(\beta)\leq \bigg(\frac{\lambda_v(\alpha)^{\frac{q-1}{q}}}{(\beta-\alpha)^{p-1}}\bigg)^\frac{n}{n-p}.
\]
We also assume without loss of generality that \(m(\Omega)= 1\). We wish to bound \(\lambda_v(\beta)\) by a function of \(\beta\), and then employ this pointwise estimate in \eqref{dist}. To this end we find successive approximations of \(\lambda_v\) by first defining \(\lambda_0(\beta)=m(\Omega)=1\), and for \(k\geq 1\) recursively defining
\[
    \lambda_{k+1}(\beta)=\inf_{0\leq\alpha<\beta}\bigg(\frac{\lambda_k(\alpha)^{\frac{q-1}{q}}}{(\beta-\alpha)^{p-1}}\bigg)^\frac{n}{n-p}.
\]
It follows by induction and an application of Lemma \ref{distbound} that \(\lambda_v(\beta)\leq \lambda_k(\beta)\) for each \(k\in\mathbb{N}\), and in particular we have for all \(\beta\geq 0\) that
\[
    \lambda_v(\beta)\leq \min\bigg\{1,\lim_{k\rightarrow\infty}\lambda_k(\beta)\bigg\}.
\]

To estimate the right-hand, we make a crude (but adequate) approximation by choosing \(\alpha=\frac{\beta}{2}\) in the infimum defining \(\lambda_k(\beta)\), so that
\begin{equation}\label{recur}
    \lambda_{k+1}(\beta)\leq \bigg(\frac{2}{\beta}\bigg)^{(p-1)(\frac{n}{n-p})}\lambda_k\bigg(\frac{\beta}{2}\bigg)^{(\frac{q-1}{q})(\frac{n}{n-p})}.
\end{equation}
Fixing \(\ell=\frac{(q-1)n}{q(n-p)}\), we now assume that \(q<\frac{n}{p}\) to ensure \(\ell<1\), and we argue by induction that
\[
    \lambda_k(\beta)\leq \bigg(\frac{2}{\beta}\bigg)^{\textstyle\frac{(p-1)q}{q-1}\sum\limits_{j=1}^k \ell^j}
\]
for each \(k\in\mathbb{N}\). The base case follows at once from taking \(k=0\) in \eqref{recur} and using that \(\lambda_0(\beta)=1\) for all \(\beta\geq 0\). For the inductive step we suppose that the claimed estimate holds up to \(k\) and we observe that
\[
    \lambda_{k+1}(\beta)\leq \bigg(\frac{2}{\beta}\bigg)^{\textstyle(p-1)(\frac{n}{n-p})}\lambda_k\bigg(\frac{\beta}{2}\bigg)^{\textstyle\ell}\leq \bigg(\frac{2}{\beta}\bigg)^{\textstyle\frac{(p-1)q}{q-1}\ell}\bigg(\frac{2}{\beta}\bigg)^{\textstyle\frac{(p-1)q}{q-1}\ell\sum\limits_{j=1}^k \ell^j}=\bigg(\frac{2}{\beta}\bigg)^{\textstyle\frac{(p-1)q}{q-1}\sum\limits_{j=1}^{k+1} \ell^j}.
\]
Thus the claimed estimate holds for each \(k\in\mathbb{N}\). Since we can take \(k\) as large as we like and \(\ell<1\), the series in the power converges to \(\frac{\ell}{1-\ell}\) as \(k\rightarrow\infty\). It follows that
\[
    \lim_{k\rightarrow\infty}\lambda_k(\beta)\leq \bigg(\frac{2}{\beta}\bigg)^{\textstyle\frac{(p-1)q}{q-1}(\frac{\ell}{1-\ell})}=\bigg(\frac{2}{\beta}\bigg)^{\textstyle\frac{(p-1)qn}{n-qp}}.
\]

Altogether then, we find the following estimate on the distribution function \(\lambda_b\) for  \(\beta\geq 0\):
\[
    \lambda_v(\beta)\leq \min\bigg\{1,\bigg(\frac{2}{\beta}\bigg)^{\textstyle\frac{(p-1)qn}{n-qp}}\bigg\}.
\]
Using this estimate we now study \(L^r(\Omega)\) regularity of \(v\) and \(u\). From identity \eqref{dist} we observe that
\[
    \int_\Omega|v|^{\textstyle r}dx=r\int_0^\infty\beta^{\textstyle r-1}\lambda_v(\beta)d\beta\leq C+C\int_2^\infty\beta^{\textstyle r-1-\frac{(p-1)qn}{n-qp}}d\beta.
\]
The integral on the right-hand side converges whenever \(1\leq r<\frac{(p-1)qn}{n-qp}\), showing that if \(r\) is within that range then \(\|v\|_r\leq C(n,p,q,r)\). Recalling the definition of \(v\), it follows that for the same \(r\) we have
\[
    \|u\|_r\leq C\|f\|_q^\frac{1}{p-1},
\]
giving the claimed bound when \(q<\frac{n}{p}\).

It remains to treat the case \(q=\frac{n}{p}\). In this instance we have \(\ell=1\) and our recursive definition for \(\lambda_k\) introduced above reads
\[
    \lambda_{k+1}(\beta)=\inf_{0\leq\alpha<\beta}\frac{\lambda_k(\alpha)}{(\beta-\alpha)^{(p-1)(\frac{n}{n-p})}}.
\]
We claim that thanks to this recursion, the following bound for \(\lambda_k\) holds for each \(k\in\mathbb{N}\):
\[
    \lambda_k(\beta)\leq \bigg(\frac{k^{k}}{\beta^{k}}\bigg)^{(p-1)(\frac{n}{n-p})}.
\]
Once again the base case follows immediately from the recursive definition of \(\lambda_1\) and the fact that \(\lambda_0(\beta)=1\). Assuming \(\lambda_k\) satisfies the claimed estimate, we observe that
\[
    \lambda_{k+1}(\beta)=\inf_{0\leq\alpha<\beta}\frac{\lambda_k(\alpha)}{(\beta-\alpha)^{(p-1)(\frac{n}{n-p})}}\leq \inf_{0\leq\alpha<\beta}\bigg(\frac{k^k}{\alpha^k(\beta-\alpha)}\bigg)^{(p-1)(\frac{n}{n-p})}.
\]
Since the function \(\alpha\mapsto\alpha^k(\beta-\alpha)\) attains a maximum of \(\frac{k^k\beta^{k+1}}{(k+1)^{k+1}}\) on \([0,\infty)\) at \(\alpha=\frac{k\beta}{k+1}\), the infimum on the right-hand side above can be calculated to give the estimate
\[
    \lambda_{k+1}(\beta)\leq \inf_{0\leq\alpha<\beta}\bigg(\frac{k^k(k+1)^{k+1}}{k^k\beta^{k+1}}\bigg)^{(p-1)(\frac{n}{n-p})},
\]
giving the required bound for \(k\in\mathbb{N}\). It follows now that for each \(\beta\geq 0\) we have
\[
    \lambda_v(\beta)\leq \min\bigg\{1,\min_{k\in\mathbb{N}}\bigg(\frac{k^{k}}{\beta^{k}}\bigg)^{(p-1)(\frac{n}{n-p})}\bigg\}.
\]
Indeed, for each fixed \(k\geq 1\) and \(\beta\) belonging to the interval \(I_k=\big[\frac{k^{k}}{(k-1)^{k-1}},\frac{(k+1)^{k+1}}{k^k}\big)\) this minimum is achieved by the function
\[
    \bigg(\frac{k^{k}}{\beta^{k}} \bigg)^{(p-1)(\frac{n}{n-p})}.
\]
Consequently we have the following pointwise bound for our distribution function:
\[
    \lambda_v(\beta)\leq \chi_{[0,1)}+\sum_{k=1}^\infty\chi_{I_k}\bigg(\frac{k^{k}}{\beta^{k}} \bigg)^{(p-1)(\frac{n}{n-p})}.
\]
Applying this estimate and using Fatou's lemma to interchange a limit and integral, we see that
\[
    \int_\Omega|v|^rdx\leq C+C\int_0^\infty\frac{\beta^r}{\beta}\bigg(\sum_{k=1}^\infty\chi_{I_k}\bigg(\frac{k^{k}}{\beta^{k}} \bigg)^{(p-1)(\frac{n}{n-p})}\bigg)d\beta\leq C+C\sum_{k=1}^\infty\int_{I_k}\frac{\beta^r}{\beta}\bigg(\frac{k^{k}}{\beta^{k}} \bigg)^{(p-1)(\frac{n}{n-p})}d\beta.
\]
Each \(I_k\) has a length of no more than three (indeed, \(m(I_k)\rightarrow e\) as \(k\rightarrow\infty\)), and on each \(I_k\) the integrand attains a maximum value at the left endpoint. Evaluating at this point we find that
\[
    \int_\Omega|v|^rdx\leq C+C\sum_{k=1}^\infty\left(\frac{k^k}{(k-1)^{k-1}}\right)^{r-1}\left(\frac{k-1}{k}\right)^{\left(k^{2}-k\right){(p-1)(\frac{n}{n-p})}}\leq C\sum_{k=1}^\infty k^{r-1}e^{-k{(p-1)(\frac{n}{n-p})}}.
\]
For any fixed \(r<\infty\) the series on the right converges, implying that there exists a constant \(C\) for which \(\|v\|_r\leq C\) and \(\|u\|_r^{p-1}\leq C\|f\|_q\). \hfill\qedsymbol{}

We remark that the final estimate does not imply that \(u\in L^\infty(\Omega)\), since the constant \(C\) may blow up as \(r\rightarrow\infty\) for some solutions.

\section{Sharpness of Theorem 1}

Finally, we verify that the exponent \(\frac{(p-1)qn}{n-pq}\) appearing in Theorem \ref{main} cannot be increased.

\begin{lem}
Let \(\varepsilon>0\) be given and assume that \(q<\frac{n}{p}\) for \(n\geq 3\) and \(p>1\). Fix \(\Omega=B(0,1)\subset\mathbb{R}^n\). There exist functions \(f\) and \(u\) satisfying \eqref{p-Poisson} such that \(f\in L^q(\Omega)\) and 
\begin{equation}\label{blowup}
    u\not\in L^{\frac{(p-1)qn}{n-pq}+\varepsilon}(\Omega).
\end{equation}
\end{lem}

\begin{proof} 
Given \(\varepsilon>0\), set \(\ell=\frac{\varepsilon(\frac{n}{q}-p)^2}{(p-1)n+\varepsilon(\frac{n}{q}-p)}>0\) and \(f(x)=-|x|^{\ell-\frac{n}{q}}\). Since \(f\) is radial, we have
\[
    \int_{B(0,1)}|f(x)|^qdx=|S^{n-1}|\int_0^1 r^{n-1}|f(r)|^qdr=|S^{n-1}|\int_0^1 r^{\ell q-1}dr.
\]
The integral on the right converges since \(\ell q>0\), meaning that \(f\in L^q(\Omega)\). Next we define
\[
    u(x)=\frac{(p-1)q^\frac{p}{p-1}}{(nq-n+\ell q)^\frac{1}{p-1}(pq-n+\ell q)} (|x|^\frac{pq-n+\ell q}{(p-1)q}-1).
\]
It is easily verified by direct differentiation that \(-\mathrm{{div}} (| \nabla u|^{p-2}\nabla u)=-|x|^{\ell-\frac{n}{2}}\) for \(x\in B(0,1)\), and if \(x\in S^{n-1}\) then \(u(x)=0\), meaning that \(u\) solves \eqref{p-Poisson}.

It remains to study the integrability of \(u\). From our choice of the constant \(\ell\) it follows that
\[
    \frac{(p-1)qn}{n-pq-\ell q}=\frac{(p-1)qn}{n-pq}+\varepsilon,
\]
and since \(u\) is a radial function on the ball we can integrate in polar coordinates to see that
\[
    \int_{B(0,1)}|u(x)|^{\frac{(p-1)qn}{n-pq}+\varepsilon}dx\geq C\bigg(\int_0^1 r^{n-1+\frac{pq-n+\ell q}{(p-1)q}(\frac{(p-1)qn}{n-pq}+\varepsilon)}dr-1\bigg)=C\bigg(\int_0^1 \frac{dr}{r}-1\bigg)=\infty.
\]
This gives \eqref{blowup}, showing that our example has the claimed properties and Theorem \ref{main} is sharp.
\end{proof}

\bibliographystyle{plain}
\pagestyle{plain}
\bibliography{references}
\end{document}